\documentclass[11pt,oneside, 
]{amsart}
\usepackage{fouriernc} 

\usepackage{Definitions}    
\usepackage{PageSetup}      
\usepackage{Environments}   

\usepackage{tikz}
\usepackage{tikz-3dplot}
\usetikzlibrary{positioning}

\usetikzlibrary{shapes.geometric}

\usetikzlibrary{arrows,decorations.pathmorphing,backgrounds,fit}

\usepackage{caption}
\usepackage{subcaption}
\usepackage{mathtools}
\ifdraft{

\usepackage{DraftSetup}

} 
{}

\newcommand{\dwsquare}{\begin{tikzpicture}
\def\x{.5}
\def\a{35}
\draw[line width=.9mm] (-\x,-\x)to[out=\a,in=180-\a](\x,-\x);
\draw[line width=.9mm](\x,-\x)to[out=90+\a,in=-90-\a](\x,\x);
\draw[line width=.9mm](\x,\x)to[out=-180+\a,in=-\a](-\x,\x);
\draw[line width=.9mm](-\x,\x)to[out=-90+\a,in=90-\a](-\x,-\x);
\draw[line width=.9mm](-\x,-\x)--(\x,\x);
\end{tikzpicture}
}

\newcommand{\dsquare}{\begin{tikzpicture}
\def\x{.5}
\def\a{25}
\draw[line width=.9mm] (-\x,-\x)--(\x,-\x);
\draw[line width=.9mm](\x,-\x)--(\x,\x);
\draw[line width=.9mm](\x,\x)--(-\x,\x);
\draw[line width=.9mm](-\x,\x)--(-\x,-\x);
\draw[line width=.9mm](-\x,-\x)--(\x,\x);
\end{tikzpicture}
}

\DeclareMathOperator{\cocyl}{Cocyl}
\DeclareMathOperator{\cyl}{Cyl}
\newcommand{\xto}[1]{\xrightarrow{#1}}
\usepackage{ wasysym }
\newcommand{\wor}[2]{{#1}\resizebox{.25cm}{!}{\dwsquare}{#2}}

\newcommand{\cgm}[2]{\lceil{#1},{#2} \rceil} 

\newcommand{\lefts}{A}
\newcommand{\leftb}{B}
\newcommand{\lma}{i}
\newcommand{\tart}{X}
\newcommand{\tarb}{Y}
\newcommand{\tma}{\alpha}
\newcommand{\tmas}{\beta} 
\newcommand{\alefts}{A} 

\newcommand{\atarb}{Z}
\newcommand{\atma}{\beta}

\newcommand{\shl}[6]{
\raisebox{-.12cm}{\resizebox{.55cm}{!}{\dwsquare}}\hspace{-.05cm}{#6}} 

\newcommand{\slnt}[5]{
\raisebox{-.12cm}{\resizebox{.55cm}{!}{\dsquare}}\hspace{-.0cm}{#5}} 
\newcommand{\slt}[3]{
\raisebox{-.12cm}{\resizebox{.55cm}{!}{\dsquare}}\hspace{-.0cm}{#3}
} 

\usepackage{dashbox}

\hypersetup{
 pdfkeywords={latex starter,template},
 pdfauthor={Niles Johnson},
}

\usepackage{tikz-cd} 


\title[Obstruction theory]{Obstruction theory in a model category and Klein and Williams' intersection invariants
}

\author{Kate Ponto}
\address{Department of Mathematics,
University of Kentucky}
\email{kate.ponto@uky.edu}

\date{\today}

\begin{document}

\begin{abstract}
We give an obstruction theory for lifts and extensions in a  model category inspired by Klein and Williams' work on intersection theory.
In contrast to the familiar obstructions from algebraic topology, this theory  produces a single invariant that is complete in the presences of the appropriate generalizations of dimension and connectivity assumptions.   
\end{abstract}

\maketitle

\setcounter{tocdepth}{1}
\tableofcontents

\section{Introduction}
Algebraic topology has a well developed theory of obstructions to lifts and extensions of maps that reduces these questions to the vanishing of cohomology classes \cite{spanier,whitehead}.  For example, a diagram 
\[\xymatrix{&E\ar[d]^p
\\
B\ar[r]^-f&Y}\]
where $p$ is a fibration defines classes in $H^{j+1}(B,\pi_j(p^{-1}(\ast)))$. If $B$  is a finite dimensional CW complex the vanishing of these invariants implies the existence of a lift of $f$.

In their work on obstructions to removing  intersections \cite{kw}, Klein and Williams give an alternative obstruction for lifts.
Let $M(\tma,\tma)$ be the double mapping cylinder (homotopy pushout) of a map $\tma\colon \tart\to \tarb$. 
There is an inclusion map 
\[\tarb\amalg \tarb\to M(\tma,\tma).\]  A map $f\colon \leftb\to \tarb$ defines a map 
\[\chi\colon \leftb\amalg \leftb\xrightarrow{f\amalg f} \tarb\amalg \tarb\to M(\tma,\tma).\]

\begin{thm}[\cite{kw}]\label{thm:kw} 
If a map $f\colon \leftb\to \tarb$ has a lift to $\tart$ up 
to homotopy, there is an extension of  $\chi$
to a map \[\cyl(\leftb)\to M(\tma,\tma).\]

Suppose $\leftb$ is a CW complex of dimension less than $2n$, $\tma\colon \tart\to \tarb$ is $n$-connected, and 
$\chi$
extends to a map \[\cyl(\leftb)\to M(\tma,\tma).\] Then $f\colon \leftb\to \tarb$ has a lift to $\tart$ up 
to homotopy.
\end{thm}

The map $\chi$ is a generalization of a classical fixed point invariant, the Reidemeister trace \cite{brown,husseini}, and  is closely related to Hatcher and Quinn's work \cite{hq} on intersections of submanifolds. 

\begin{rmk}
The dimension and connectivity hypotheses in \cref{thm:kw}
are reminiscent of assumptions 
used to imply the classical obstructions take values in trivial groups.  The ranges of connectivity and dimension in \cref{thm:kw} do not force this triviality.  
\end{rmk}

The fact that 
\cref{thm:kw} was proven for topological spaces  is an artifact of the motivating example.  
In this paper we prove a significant generalization of \cref{thm:kw} for model categories that incorporates extension questions and 
makes the roles of various assumptions transparent.  We work in model categories in this paper since they provide a convenient way to describe the necessary hypotheses, but I expect this argument can be easily adapted to closely related environments.

Our motivating question is the following:
\begin{question}\label{ques:motivating}
Given objects and solid arrows  so the diagram in  \eqref{fig:intro_help} commutes, what conditions imply dotted maps exist so the entire diagram commutes?
\begin{equation}
          \xymatrix{\lefts\ar[dd]^{\lma}\ar[rr]^-{i_0}&&\cyl(\lefts)\ar[dd]|\hole\ar[dl]_-H&&\lefts\ar[dd]^{\lma}\ar[ll]_{i_1}\ar[dl]_{h}
\\
&\tarb&&\tart\ar[ll] _(.35)\tma
\\
\leftb\ar[ur]^f\ar[rr]^-{i_0}&&\cyl(\leftb)\ar@{.>}[ul]_{K}&&\leftb\ar[ll]_{i_1}\ar@{.>}[ul]_{{g}}}
\label{fig:intro_help}
\end{equation}
\end{question}

The diagram in \eqref{fig:intro_help} is the  HELP diagram \cite[10.3]{concise}. 
To connect this question back to \cref{thm:kw}, let $\alefts$ be the empty space (initial object).  Then $g$ is the lift of $f$ up to the homotopy $K$.
\begin{notation}
Let $\shl{i}{\alpha}{f}{H}{h}{\eqref{fig:intro_help}}$ denote the set of pairs $(K,g)$ so that the diagram in \eqref{fig:intro_help}
commutes. 
We write $\wor{i}{\alpha}$
 if for all choices of $f$, $H$, and $h$ the set  $\shl{i}{\alpha}{f}Hh{\eqref{fig:intro_help}}$ is nonempty.  
 Similarly, let $\slnt{i}{\alpha}{g}{\ell}{\eqref{intro:notation_strict_lift}}$ be the set of strict lifts in the commutative diagram in  \eqref{intro:notation_strict_lift}.
\begin{equation}\label{intro:notation_strict_lift} \xymatrix{
A\ar[r]^-\xi\ar[d]^-i&X\ar[d]^\alpha\\
B\ar@{.>}[ur]\ar[r]^-f&Y}
\end{equation}
We write $\lma\resizebox{.25cm}{!}{\dsquare}\tma$  if for all choices of $f$ and $\xi$ the set $\slnt{i}{\alpha}{g}{\ell}{\eqref{intro:notation_strict_lift}}$  is nonempty.
\end{notation}

Our main result has three parts that mimic those of \cref{thm:kw}.  Informally they are:
\begin{enumerate}
    \item given the solid arrows in the diagram in  \eqref{fig:intro_help}, there is  a map $\chi$ generalizing that in \cref{thm:kw},
    \item 
    
    if $\shl{i}{\alpha}{f}{H}{h}{\eqref{fig:intro_help}}$ is nonempty then  $\chi$ is trivial, and 
    \item   
    
    if $\chi$ is trivial (and additional hypotheses are satisfied) then $\shl{i}{\alpha}{f}{H}{h}{\eqref{fig:intro_help}}$ is nonempty.
\end{enumerate}

To give a more formal statement of our main result, 
let 
\begin{equation}\label{eq:afib_fact}\tart \xto{c(\tma)}F(\tma)\xto{\tilde{f}(\tma)} \tarb
\end{equation}
be a factorization of $\tma\colon \tart\to \tarb $ as a cofibration and acyclic fibration.  
If $\ast$ is the terminal object, we have the following vacuously commutative diagram 
\begin{equation}\label{fig:diagram_with_terminal}
    \xymatrix{A\ar[d]^i\ar[r]^-{h}&X\ar[r]^-{c(\alpha)}&F(\alpha)\ar[d]
    \\B\ar[rr]&&\ast}
\end{equation}
    Let $N(i)$ be the double mapping cylinder of $i$ with itself and $M(\alpha,\alpha)$ be the homotopy pushout of $\alpha$ with itself.  (The obvious asymmetry in these choices is addressed in \cref{rmk:n_v_m}.)
    
\begin{thm}\label{thm:intro_model_cat}
Suppose given the solid arrows in the commutative diagram in \eqref{fig:intro_help}. 
\begin{enumerate}
    \item  \label{it:thm_kw_def} 
Each  $\hat{h}\in \slt{i}{c(\alpha)\circ h}{\eqref{fig:diagram_with_terminal}}$ defines a map 
\[\chi_{\hat{h}}\colon N(i)\to M(\alpha,\alpha)\] generalizing the map $\chi$ of \cref{thm:kw}.
\item \label{it:thm_kw_forward} 
For each pair $(K,g)\in \shl{i}{\alpha}{f}Hh{\eqref{fig:intro_help}}$ there is a $\hat{h}\in \slt{i}{c(\alpha)\circ h}{\eqref{fig:diagram_with_terminal}}$ so that  $\chi_{\hat{h}}$ extends over the cylinder on $B$.

\item \label{it:thm_kw_backward} 

Under appropriate generalizations of the dimension and connectivity conditions  (see \cref{thm:full_hy}), 
$\shl{i}{\alpha}{f}Hh{\eqref{fig:intro_help}}$ is nonempty if there is an 
 $\hat{h}\in \slt{i}{c(\alpha)\circ h}{\eqref{fig:diagram_with_terminal}}$ so that $\chi_{\hat{h}}$ extends over the cylinder on $B$.
\end{enumerate}
\end{thm}

I came to this project from the paper of Klein and Williams \cite{kw} and related work focused on fixed point theory \cite{coufal,sun}  and have not found much existing literature related to obstructions in this generality.  The closest seems to be the paper of Christensen, Dwyer and Isaksen \cite{CDI} but there does not seem to be an immediate translation between perspectives. 

\subsection*{For the reader familiar with Klein and Williams's work} After defining the map $\chi$ for topological spaces, Klein and Williams \cite{kw} show the corresponding stable invariant is only trivial when $\chi$ is trivial and they use duality to give an alternative description of their stable invariant.  Those steps don't make sense for the generality considered here and so are omitted.

\subsection*{Organization} In \S\ref{sec:euler_class} we define the map $\chi$ and prove  \cref{thm:intro_model_cat}\ref{it:thm_kw_forward}. 
In \S\ref{sec:c_gap_map} we  fix notation and recall the cartesian gap map and the Blakers-Massey theorem. In \S\ref{sec:main_proof} we prove  \cref{thm:intro_model_cat}\ref{it:thm_kw_backward}. In \S\ref{sec:model} we finish up some model category business.

\subsection*{Acknowledgements}
This paper evolved from conversations with Inbar Klang, Sarah Yeakel, and Cary Malkiewich about generalizations of fixed point invariants. 
Many thanks to Inbar and Sarah for comments on an earlier version of this paper.
Also thanks to Nima  Rasekh and Dan Duggar for helping untangle confusions.   Finally, thanks to the referee for careful reading and helpful suggestions.

The author was partially supported by NSF grant  DMS-1810779 and the Royster Research Professorship at the University of Kentucky.

\section{The Euler class}\label{sec:euler_class}

Working in a model category with functorial good cylinders, let 
\[\leftb\amalg \leftb\xto{i_0\amalg i_1}\cyl(\leftb)\xto{\pi}\leftb\]
be a factorization of the fold map through the functorial cylinder.  When we need to specify the domain of $i_0$ or $i_1$ we will write $i_{0,\leftb}$ or $i_{1,\leftb}$.  If $\lma\colon \lefts\to \leftb$, let $N_j(\lma)$, $j=0,1$, be the mapping cylinder of $\lma$ and $i_j$:
\begin{equation}N_j(\lma)\coloneqq \text{Pushout}(\cyl(\lefts)\xleftarrow{i_j}\lefts\xto{\lma}\leftb). \end{equation}
Let $\iota_j\colon N_j(\lma)\to \cyl(B)$ be the map induced by $\cyl(\lma)\colon \cyl(\lefts)\to\cyl(\leftb) $ and $i_{j,\leftb} $.
Let  $N(\lma)$ be the double mapping cylinder on $\lma$:  
\[N(\lma)\coloneqq \text{Pushout}(\cyl(\lefts)\xleftarrow{i_0\amalg i_1}\lefts\amalg \lefts\xto{\lma\amalg \lma}\leftb\amalg \leftb)\] and $\iota\colon N(i)\to \cyl(B)$ be induced by the universal property of the pushout.

For maps $\tma\colon \tart\to \tarb $ and $ \atma\colon \tart\to \atarb$ recall the factorization from \eqref{eq:afib_fact} and 
let
$M(\tma,\atma)$ be the pushout in the diagram in  \eqref{eq:defn_Maa}.

     \begin{equation}
         \xymatrix{\tart\ar[r]^-{c(\tma)}\ar[d]^-{\atma}&F(\tma)\ar@{.>}[d]^j
\\
\atarb\ar@{.>}[r]^-{j'}&M(\tma,\atma)}
    \label{eq:defn_Maa}
    \end{equation}

\begin{rmk}\label{rmk:n_v_m}
The choice to make $N(i)$ the double mapping cylinder and $M(\alpha,\beta)$ the homotopy pushout is  intentional.  In what follows we will need to have very explicit access to the cylinder in $N(i)$, while there will not be similar requirements  for $M(\alpha,\beta)$.
\end{rmk}

\begin{lem}[\cref{thm:intro_model_cat}\ref{it:thm_kw_def}]\label{def:chi}
There is a function 
\[\slt{i}{c(\alpha)\circ h}{\eqref{fig:diagram_with_terminal}} \xto{\chi} \mathrm{Map}(N(i),M(\alpha,\alpha)).\] 
\end{lem}

\begin{proof}
If $\hat{h}\in \slt{i}{c(\alpha)\circ h}{\eqref{fig:diagram_with_terminal}}$
then $\hat{h}$ makes 
the top square in the diagram in  \eqref{fig:def_chi} 
commute.  The remaining squares of the diagram in  \eqref{fig:def_chi} commute by the diagrams in  \eqref{fig:intro_help} and \eqref{eq:defn_Maa}. Then \[\chi_{\hat{h}}\colon N(\lma)\to M(\tma,\tma)\] is induced by the diagram in  \eqref{fig:def_chi} using the universal property of the pushout.
\end{proof}

\noindent\begin{subequations}
\begin{minipage}{.49\textwidth}
\begin{equation}
\xymatrix{
\lefts\ar[r]^\lma\ar[d]^{i_1}\ar[rd]^h&\leftb\ar[dr]^{\hat{h}}
\\
N_0(\lma)\ar[rd]_{f\amalg H}&\tart\ar[r]^{c(\tma)}\ar[d]^\tma&F(\tma)\ar[d]^j
\\
&\tarb\ar[r]^-{j'}&M(\tma,\tma)}
    \label{fig:def_chi}
    \end{equation}
    \end{minipage}
    \begin{minipage}{.49\textwidth}
\begin{equation}
\xymatrix{N(\lma)\ar[d]^{\iota}\ar[r]^-{\chi_{\hat{h}}}&M(\tma,\tma)
\\
\cyl(\leftb)\ar@{.>}[ur]_{\mathcal{X}}}
    \label{fig:def_chi_triv}
\end{equation}
    \end{minipage}
    \end{subequations}

We say $\chi_{\hat{h}}$ is {\bf trivial} if there is a map $\mathcal{X}$ as in the diagram in  \eqref{fig:def_chi_triv}.  We say $\chi$ is {\bf trivial} if the set 
\[\{\chi_{\hat{h}}\}_{\hat{h}\in \slt{i}{c(\alpha)\circ h}{\eqref{fig:diagram_with_terminal}}}.\]
is nonempty and all $\chi_{\hat{h}}$ are trivial.

\begin{rmk}
In special cases (see \cref{rmk:hathAinitial,rmk:hathsection}) there are topologically motivated assumptions that guarantee $\slt{i}{c(\alpha)\circ h}{\eqref{fig:diagram_with_terminal}}$ is nonempty. There don't appear to be similar assumptions when working at this level of generality.
\end{rmk}

\begin{prop}[\cref{thm:intro_model_cat}\ref{it:thm_kw_forward}]\label{prop:chi_triv}

There is a function 
$\shl{i}{\alpha}{f}Hh{\eqref{fig:intro_help}}\to \slt{i}{c(\alpha)\circ h}{\eqref{fig:diagram_with_terminal}} $
and the composite 
\[\shl{i}{\alpha}{f}Hh{\eqref{fig:intro_help}}\to \slt{i}{c(\alpha)\circ h}{\eqref{fig:diagram_with_terminal}} \xto{\chi} \mathrm{Map}(N(i),M(\alpha,\alpha))\] 
is trivial. 
\end{prop}

\begin{proof}
For $(K,g)\in \shl{i}{\alpha}{f}Hh{\eqref{fig:intro_help}}$, take $\hat{h}\in \slt{i}{c(\alpha)\circ h}{\eqref{fig:diagram_with_terminal}}$ to be composite \[\leftb\xto{g}\tart\xto{c(\tma)} F(\tma).\]  Then the diagram in  \eqref{fig:def_chi_triv_start} defines a map $\cyl(\leftb)\to M(\tma,\tma)$ that  extends $\chi_{\hat{h}}$.
\end{proof}
\begin{equation}
\xymatrix{\leftb\ar[r]^-\id\ar[d]^{i_1}\ar[rd]^g&\leftb\ar[dr]^{\hat{h}\coloneqq c(\alpha)\circ g}
\\
\cyl(\leftb)\ar[dr]^K&\tart\ar[r]^-{c(\tma)}\ar[d]^-\tma&F(\tma)\ar[d]^j
\\
&\tarb\ar[r]^-{j'}&M(\tma,\tma)}
    \label{fig:def_chi_triv_start}
\end{equation}

\section{The cartesian gap map}\label{sec:c_gap_map}
For a map $\tma\colon \tart\to \tarb $, let 
\begin{equation}\label{eq:acofib_fact}
    \tart\xto{\tilde{c}(\tma)} C(\tma)\xto{{f}(\tma)} \tarb 
\end{equation} be a factorization of $\tma$ as an acyclic cofibration and a fibration.
For maps $\tma\colon \tart\to \tarb $ and $ \atma\colon \tart\to \atarb$
let
$P(\tma,\atma)$ be the pullback in the diagram in  \eqref{eq:blma_pmap}.

\noindent\begin{subequations}
\begin{minipage}{.49\textwidth}
            \begin{equation}
\xymatrix{P(\tma,\atma)\ar@{.>}[r]^-k\ar@{.>}[d]^{k'}&C(j)\ar[d]^{f(j)}
\\
\atarb\ar[r]^-{j'}&M(\tma,\atma)
}
\label{eq:blma_pmap}
\end{equation}
    \end{minipage}
    \begin{minipage}{.49\textwidth}
        \begin{equation}
\xymatrix{\tart\ar@{.>}[dr]^{\cgm{\tma}{\atma}}\ar@/_/[ddr]_{\tmas}\ar[r]^-{c(\tma)}
&F(\tma)\ar@/^/[dr]^-{\tilde{c}(j)}
\\
&P(\tma,\atma)\ar[r]^-k\ar[d]^{k'}&C(j)\ar[d]^{f(j)}
\\
&\atarb\ar[r]^-{j'}&M(\tma,\atma)
}\label{eq:blma_map}
\end{equation}
    \end{minipage}
    \end{subequations}

\begin{defn}\label{def:cgm}
The {\bf cartesian gap map} of $\alpha$ and $\beta$, denoted $\cgm{\tma}{\atma}$, is the dotted map in the diagram in  \eqref{eq:blma_map} induced by the universal property of the pullback.
\end{defn}

\begin{example}\label{ex:bm} There are several important connectivity results for the cartesian gap map.

\begin{enumerate}
    \item In a stable model category, such as chain complexes of $R$-modules or orthogonal spectra,  the homotopy cartesian and cocartesian squares agree and the cartesian gap map is a weak equivalence \cite[7.1.12]{hovey}.  
    \item In the category of topological spaces, if $\tma \colon \tart\to \tarb$ is $n$-connected and $\atma \colon \tart\to \atarb$ is $n'$-connected the classical Blakers-Massey theorem \cite[6.9]{td} asserts  $\cgm{\alpha}{\beta}$ is  $(n+n'-1)$-connected.  
    
    There is a similar, but significantly more complicated, version of the Blakers-Massey theorem for spaces with an action of a finite group \cite{dotto}.
    \item 
    The recent papers \cite{AGFJ,CSW} prove generalizations of the classical Blakers-Massey theorem.   \cite{AGFJ} proves a version for  higher topoi.  

\end{enumerate}
\end{example}

The essential hypothesis in the converse of \cref{prop:chi_triv} is a lifting condition for the cartesian gap map.  The first two examples in \cref{ex:bm} imply relevant lifting conditions.

\begin{prop}\label{prop:top_bm}\hfill 
\begin{enumerate}
    \item\label{prop:top_bm1} 
    In a model category, if $A$ is cofibrant, $\iota$ is a cofibration, $P(\tma,\atma)$ is fibrant,  and  $i\resizebox{.25cm}{!}{\dsquare} (X\times_{\cgm{\alpha}{\beta}, \ev_0}\cocyl(P(\tma,\atma))\xto{\ev_1}P(\tma,\atma))$ then
\[\wor{i}{\cgm{\alpha}{\beta}}.\]
\item \label{it:top_bm}
If $\lma\colon \lefts\to \leftb$ is a relative CW complex of dimension m, $\tma \colon \tart\to \tarb$ is $n$-connected, $\atma \colon \tart\to \atarb$ is $n'$-connected and $m\leq (n+n'-1)$ then \[\wor{\lma}{\cgm{\tma}{\atma}}.\] 
\end{enumerate}
\end{prop}

\begin{proof}
We postpone the proof of the first statement to  \cref{sec:model} since it is comparatively lengthy and not illuminating for the ideas considered here.
See page \pageref{proof:top_bm}
with $\cgm{\alpha}{\beta}=\theta$.

For the second, the classical Blakers-Massey theorem implies the map 
\[\tart\to P(\tma,\tmas)\]
is $(n+n'-1)$-connected.  Then the homotopy extension and lifting property \cite[10.3]{concise} produces the required lift.
\end{proof}

\cref{prop:top_bm}\ref{it:top_bm} should have an equivariant generalization following  \cite{dotto} and this  should allow for an alternative approach to the main result in  \cite{kw2}.

\section{Proof of the main result}\label{sec:main_proof}
Recall the mapping cylinder $N_j(\lma)$  and maps $\iota_j$ from the introduction.
The maps in \eqref{fig:intro_help}, \eqref{eq:afib_fact}, and \eqref{eq:acofib_fact}  define the  commutative diagram in \eqref{fig:converse_assumption}.  
\begin{equation}
\xymatrix{
\lefts\ar[r]^-h\ar[d]^{i_1}&\tart\ar[r]^-{c(\tma)}\ar@{-->}[d]^\alpha&F(\tma)\ar@{-->}[dr]^{j}\ar[r]^-{\tilde{c}(j)}&C(j)\ar[d]^{f(j)}
\\
N_0(\lma)\ar[r]^-{H\amalg f}& \tarb\ar[rr]^-{j'}&&M(\tma,\tma)}
\label{fig:converse_assumption}
\end{equation}

\begin{lem}\label{lem:construct_phi_psi}
An element $\lambda \in \slnt{}{}{}{}{ \eqref{fig:converse_assumption}}$
defines maps $\Phi_\lambda\colon B\to P(\alpha, \alpha)$ and $\Psi_\lambda\colon \cyl(A)\to P(\alpha, \alpha)$ so the solid arrow portions of the diagram in  \eqref{fig:def_big_diagram} commute.
\end{lem}
        \begin{equation}
\xymatrix{\lefts\ar[rr]^-{i_0}\ar[dd]^\lma\ar@0{->}[dr]|(.5){R1}
&&\cyl(\lefts)\ar[dd]|\hole\ar[dl]_{\Psi_\lambda}
\ar@0{->}[dr]|(.5){R2}
&&\lefts\ar[ll]_-{i_1}\ar[dd]^\lma\ar[dl]_{h}
\\
&P(\tma,\tma)&&\tart\ar[ll]_(.3){\cgm{\tma}{\tma}}
\\
\leftb\ar[rr]^-{i_0}\ar[ur]^{\Phi_\lambda}&&\cyl(\leftb)\ar@{.>}[ul]^{\hat{\Psi}}&&\leftb\ar[ll]_{i_1}\ar@{.>}[ul]^{\hat{g}}
}
    \label{fig:def_big_diagram}
    \end{equation}

\begin{proof}
 A lift $\lambda$ in the diagram in  \eqref{fig:converse_assumption} defines maps 
\[\phi\colon \leftb\to C(j)\text{ and }\psi\colon \cyl(\lefts)\to C(j)\] so that the diagrams in \eqref{fig:lift_equiv_1}, \eqref{fig:lift_equiv_2}, \eqref{fig:lift_equiv_3}, and \eqref{fig:lift_equiv_4} commute.  

\noindent\begin{subequations}
    \begin{minipage}{.49\textwidth}
    \begin{equation}
        \xymatrix{\leftb\ar[r]^-{\phi}\ar[d]^f&C(j)\ar[d]^{f(j)}
\\
\tarb\ar[r]^-{j'}&M(\tma,\tma)}
     \label{fig:lift_equiv_1}
    \end{equation}
    \end{minipage}
        \begin{minipage}{.49\textwidth}
        \begin{equation}
\xymatrix{\cyl(\lefts)\ar[r]^-{\psi}\ar[d]^{H}&C(j)\ar[d]^{f(j)}\\\tarb\ar[r]^-{j'}&M(\tma,\tma)}
        \label{fig:lift_equiv_2}
    \end{equation}
    \end{minipage}
    
    \noindent
        \begin{minipage}{.49\textwidth}
        \begin{equation}
        \xymatrix{\lefts\ar[d]^{i_1}\ar[r]^-h&\tart\ar[r]^-{c(\tma)}&F(\tma)\ar[d]^{\tilde{c}(j)}
        \\
        \cyl(\lefts)\ar[rr]^-{\psi}&&C(j)}
        \label{fig:lift_equiv_3}
    \end{equation}
    \end{minipage}
        \begin{minipage}{.49\textwidth}
        \begin{equation}
        \xymatrix{\lefts\ar[r]^-{i_0}\ar[d]^\lma&\cyl(\lefts)\ar[d]^{\psi}
\\
\leftb\ar[r]^{\phi}&C(j)}
      \label{fig:lift_equiv_4}
    \end{equation}
    \end{minipage}
\end{subequations}

The diagram in \eqref{fig:lift_equiv_1} defines a map 
\[\Phi_\lambda\colon \leftb\to P(\tma,\tma)\] and  the diagram in  \eqref{fig:lift_equiv_2} defies a map 
\[\Psi_\lambda\colon \cyl(\lefts)\to P(\tma,\tma).\]

Since $P(\alpha,\alpha)$ is a pullback, to show the regions $R1$ and $R2$ of the diagram in  \eqref{fig:def_big_diagram} commute it is
enough to show they commute after composition  with $k$ and $k'$ from the diagram in \eqref{eq:blma_pmap}.  This gives the following four diagrams where the justification for the commutativity of each region is indicated in that region.

\noindent\begin{subequations}
        \begin{minipage}{.49\textwidth}
        \begin{equation}
        \xymatrix@C=6pt@R=6pt{A\ar[dd]_-h\ar[rr]^-{i_1}\ar@0{->}[ddddrrrr]|(.35){\eqref{fig:lift_equiv_3}}
&&\cyl(A)\ar[rr]^-{\Psi_\lambda}\ar[rrdddd]|\psi
&\ar@0{->}[ddddr]|(.35){\mathrm{Def}\, \Psi_\lambda}&P(\alpha,\alpha)\ar[dddd]^k
\\
&&&
\\
X\ar[drr]|{c(\tma)}\ar[dd]_{\cgm{\tma}{\tma}}
\\
\ar@0{->}[drrrr]|(.25){\mathrm{Def}\, \cgm{\tma}{\tma}}
&&F(\alpha)\ar[drr]|{\tilde{c}(j)}
\\
P(\alpha,\alpha)\ar[rrrr]_-k
&&&&C(j)}\label{fig:lift_equiv_next_5}
\end{equation}
\end{minipage}
        \begin{minipage}{.49\textwidth}
        \begin{equation}
\xymatrix@C=6pt@R=6pt{A\ar[dd]_-h\ar[rr]^-{i_1}\ar@0{->}[ddddrrrr]|(.35){\eqref{fig:intro_help}}
&&\cyl(A)\ar[rr]^-{\Psi_\lambda}\ar[rrdddd]|H
&\ar@0{->}[ddddr]|(.35){\mathrm{Def}\, \Psi_\lambda}&P(\alpha,\alpha)\ar[dddd]^{k'}
\\
&&&
\\
X\ar[ddrrrr]|{\tma}\ar[dd]_{\cgm{\tma}{\tma}}
\\
\ar@0{->}[drrrr]|(.25){\mathrm{Def}\, \cgm{\tma}{\tma}}
&&
\\
P(\alpha,\alpha)\ar[rrrr]_-{k'}
&&&&Y}\label{fig:lift_equiv_next_6}
\end{equation}
\end{minipage}

  \noindent      \begin{minipage}{.49\textwidth}
        \begin{equation}
\xymatrix@C=6pt@R=6pt{A\ar[dd]_-i\ar[rr]^-{i_0}\ar@0{->}[ddddrrrr]|(.35){\eqref{fig:lift_equiv_4}}
&&\cyl(A)\ar[rr]^-{\Psi_\lambda}\ar[rrdddd]|\psi
&\ar@0{->}[ddddr]|(.35){\mathrm{Def}\, \Psi_\lambda}&P(\alpha,\alpha)\ar[dddd]^{k}
\\
&&&
\\
B\ar[ddrrrr]|{\phi}\ar[dd]_{\Phi_\lambda}
\\
\ar@0{->}[drrrr]|(.25){\mathrm{Def}\, \Phi_\lambda}
&&
\\
P(\alpha,\alpha)\ar[rrrr]_-{k}
&&&&C(j)}\label{fig:lift_equiv_next_7}
\end{equation}
\end{minipage}
        \begin{minipage}{.49\textwidth}
        \begin{equation}
\xymatrix@C=6pt@R=6pt{A\ar[dd]_-i\ar[rr]^-{i_0}\ar@0{->}[ddddrrrr]|(.35){\eqref{fig:intro_help}}
&&\cyl(A)\ar[rr]^-{\Psi_\lambda}\ar[rrdddd]|H
&\ar@0{->}[ddddr]|(.35){\mathrm{Def}\, \Psi_\lambda}&P(\alpha,\alpha)\ar[dddd]^{k'}
\\
&&&
\\
B\ar[ddrrrr]|{f}\ar[dd]_{\Phi_\lambda}
\\
\ar@0{->}[drrrr]|(.25){\mathrm{Def}\, \Phi_\lambda}
&&
\\
P(\alpha,\alpha)\ar[rrrr]_-{k'}
&&&&Y}\label{fig:lift_equiv_next_8}
\end{equation}
\end{minipage}

        \end{subequations}
Then the diagrams in  \eqref{fig:lift_equiv_next_5} and \eqref{fig:lift_equiv_next_6} imply region $R2$ of the diagram in \eqref{fig:def_big_diagram} commutes and the diagrams in  \eqref{fig:lift_equiv_next_7} and \eqref{fig:lift_equiv_next_8} imply region $R1$ of the diagram in \eqref{fig:def_big_diagram} commutes.
\end{proof}

\begin{prop}\label{thm:min_hyp_statement}
For each $\lambda \in \slnt{}{}{}{}{ \eqref{fig:converse_assumption}}$, there is a function
\[
\shl{i}{\cgm{\tma}{\tma}}{\Phi_\lambda}{\Psi_\lambda}h{\eqref{fig:def_big_diagram}}
\to \shl{i}{\alpha}{f}Hh{\eqref{fig:intro_help}}\]
\end{prop}

\begin{proof}
If $(\hat{\Psi},\hat{g})\in \shl{i}{\cgm{\tma}{\tma}}{\Phi_\lambda}{\Psi_\lambda}h{\eqref{fig:def_big_diagram}}$ then 
$(k'\circ \hat{\Psi},\hat{g})$ is an element of $\shl{i}{\alpha}{f}Hh{\eqref{fig:intro_help}}$.  See the diagram in  \eqref{fig:extended_diagram}.
\end{proof}
        \begin{equation}
\xymatrix{\lefts\ar[rrr]^-{i_0}\ar[dd]^\lma&&&\cyl(\lefts)\ar[dd]|\hole\ar[dl]_{\Psi_\lambda}\ar@/_/[dll]_-{H}&&\lefts\ar[ll]_-{i_1}\ar[dd]^\lma\ar[dl]_{h}
\\
&\tarb&\ar[l]_{k'}P(\tma,\tma)&&\tart\ar[ll]_(.3){\cgm{\tma}{\tma}}
\\
\leftb\ar[rrr]^-{i_0}\ar[ur]^f\ar@/_/[urr]^{\Phi_\lambda}&&&\cyl(\leftb)\ar@{.>}[ul]^{\hat{\Psi}}&&\leftb\ar[ll]_{i_1}\ar@{.>}[ul]^{\hat{g}}
}\label{fig:extended_diagram}
    \end{equation}

We now compare a lift in the diagram in  \eqref{fig:converse_assumption} to the vanishing of $\chi_{\hat{h}}$.

\begin{lem}\label{lem:lift_conditions}
Suppose given the solid arrows in the diagram in  \eqref{fig:intro_help},  and assume that $i_{1,A} 
\resizebox{.25cm}{!}{\dsquare} f(j)$ and
$\iota_1 \resizebox{.25cm}{!}{\dsquare} f(j)$. 
 There is a function 
 \[\ell\colon \mathrm{ker}(\chi)\to\slnt{}{}{}{}{ \eqref{fig:converse_assumption}}.\]
\end{lem}

\begin{proof} Since $i_{1,A} 
\resizebox{.25cm}{!}{\dsquare}f(j)$,
there is a lift $\lambda$ in the diagram in  \eqref{fig:def_b}.
\begin{equation}
        \xymatrix{\lefts\ar[r]^-h\ar[d]^{i_1}&\tart\ar@{-->}[d]^\tma\ar[r]^{c(\tma)}&F(\tma)\ar@{-->}[dr]^j
\ar[r]^{\tilde{c}(j)}&C(j)\ar[d]^{f(j)}
\\
\cyl(\lefts)
\ar[r]^-H&\tarb\ar[rr]^-{j'}&&M(\tma,\tma)}
    \label{fig:def_b}
\end{equation}
    Then the diagram in \eqref{fig:n1_pullback}
    commutes.  A lift in the diagram in  \eqref{fig:n1_pullback} and $c(\alpha)\circ \hat{h}$ define a map $\mu\colon N_1(\lma)\to C(j)$.  Using this map,  the diagram in  
    \eqref{fig:def_c_3} commutes.  Since $\iota_1 \resizebox{.25cm}{!}{\dsquare} f(j)$, the diagram in  \eqref{fig:def_c_3}  has a lift $\nu$. 
    
\noindent
\begin{subequations}
    \begin{minipage}{.49\textwidth}
    \begin{equation}
        \xymatrix{\lefts\ar[dd]^{i_1}\ar[rr]^\lma\ar[dr]^h&&\leftb\ar[d]^{\hat{h}}
    \\&\tart\ar[r]^-{c(\tma)}&F(\tma)\ar[d]^{\tilde{c}(j)}
    \\
    \cyl(\lefts)\ar[rr]^-\lambda&&C(j)
    }
 \label{fig:n1_pullback}
    \end{equation}
    \end{minipage}
        \begin{minipage}{.49\textwidth}
    \begin{equation}
        \xymatrix{N_1(i)\ar[r]^-{\mu}\ar[d]&C(j)\ar[dd]^{f(j)}
 \\
 N(\lma)\ar[dr]^-{\chi_{\hat{h}}}\ar[d]
 \\
 \cyl(\leftb)\ar@{.>}[uur]^{\nu}\ar[r]^-{\mathcal{X}}&M(\tma,\tma)}
    \label{fig:def_c_3}
    \end{equation}
    \end{minipage}
\end{subequations}

Expanding the diagram in \eqref{fig:def_c_3}, the diagrams in \eqref{fig:def_c_4} and \eqref{fig:def_c_5} commute.

\noindent
\begin{subequations}
    \begin{minipage}{.49\textwidth}
    \begin{equation}
        \xymatrix{
\lefts\ar[dd]^\lma\ar[r]^-{i_0}&\cyl(\lefts)\ar[dr]^{\lambda}\ar[d]
\\
&N_1(\lma)\ar[r]^{\mu}\ar[d]&C(j)
\\
\leftb\ar[r]^-{i_0}&\cyl(\leftb)\ar[ru]^-{\nu}}
    \label{fig:def_c_4}
    \end{equation}
    \end{minipage}
    \begin{minipage}{.49\textwidth}
    \begin{equation}
        \xymatrix{
\leftb\ar[d]^f\ar[r]^-{i_0}&\cyl(\leftb)\ar[dr]^{\mathcal{X}}\ar[r]^-{\nu}&C(j)\ar[d]^{f(j)}
\\
\tarb\ar[rr]^{j'}&&M(\tma,\tma)}
    \label{fig:def_c_5}
    \end{equation}
    \end{minipage}
\end{subequations}

\noindent The diagram in \eqref{fig:def_c_4} shows that $B\xto{i_0}\cyl(B)\xto{\nu} C(j)$ and $\lambda$ define a map $N_0(\lma)\to C(j)$.  To verify this is a lift for the diagram in  \eqref{fig:converse_assumption}, notice that the restriction to $\cyl(A)$ commutes by the diagram in  \eqref{fig:def_b} and the restriction to $B$ commutes by the diagram in  \eqref{fig:def_c_5}.   
\end{proof}

Using \cref{lem:construct_phi_psi,lem:lift_conditions}, $\hat{h}\in \mathrm{ker}(\chi)$ gives the following (minor) variation on the diagram in  \eqref{fig:def_big_diagram}.
        \begin{equation}
\xymatrix{\lefts\ar[rr]^-{i_0}\ar[dd]^\lma
&&\cyl(\lefts)\ar[dd]|\hole\ar[dl]_{\Psi_{\ell(\hat{h})}}
&&\lefts\ar[ll]_-{i_1}\ar[dd]^\lma\ar[dl]_{h}
\\
&P(\tma,\tma)&&\tart\ar[ll]_(.3){\cgm{\tma}{\tma}}
\\
\leftb\ar[rr]^-{i_0}\ar[ur]^{\Phi_{\ell(\hat{h})}}&&\cyl(\leftb)\ar@{.>}[ul]^{\hat{\Psi}}&&\leftb\ar[ll]_{i_1}\ar@{.>}[ul]^{\hat{g}}
}
    \label{fig:def_big_diagram_ell}
    \end{equation}

The following result is 
an immediate consequence of  \cref{thm:min_hyp_statement,lem:lift_conditions}.  
\begin{thm}[\cref{thm:intro_model_cat}.\ref{it:thm_kw_backward}]\label{thm:full_hy}
Suppose given the solid arrows in the diagram in  \eqref{fig:intro_help}  and assume  
$i_{1,A} 
\resizebox{.25cm}{!}{\dsquare} f(j)$ and
$\iota_1 \resizebox{.25cm}{!}{\dsquare} f(j)$.

If there is an  $\hat{h}\in \slt{i}{c(\alpha)\circ h}{\eqref{fig:diagram_with_terminal}}$ so that 
 $\chi_{\hat{h}}$ is trivial and $\shl{i}{\cgm{\tma}{\tma}}{\Phi_{\ell(\hat{h})}}{\Psi_{\ell(\hat{h})}}h{\eqref{fig:def_big_diagram_ell}}$
 is nonempty then $\shl{i}{\alpha}{f}{H}{h}{\eqref{fig:intro_help}}$ is nonempty.
\end{thm}

\begin{rmk}
The hypotheses $i_{1,A} 
\resizebox{.25cm}{!}{\dsquare} f(j)$ and
$\iota_1 \resizebox{.25cm}{!}{\dsquare} f(j)$ can be satisfied by assuming $A$ is cofibration (so $i_{1,A}$ is a cofibration) and $\iota_1$ is an acyclic cofibration.  These parallel the assumption that $B$ is a CW complex in \cref{thm:kw}.
\end{rmk}

Note that this theorem requires $\slt{i}{c(\alpha)\circ h}{\eqref{fig:diagram_with_terminal}}$ to be nonempty, so in the following result 
we include that assumption explicitly.

\begin{cor}\label{cor:full_hy} For maps $i\colon A\to B$ and $\alpha\colon X\to Y$ assume 
\begin{itemize}
    \item $i_{1,A} 
\resizebox{.25cm}{!}{\dsquare} f(j)$,
$\iota_1 \resizebox{.25cm}{!}{\dsquare} f(j)$, $i\resizebox{.25cm}{!}{\dsquare}(F(\alpha)\to \ast)$ and  
\item $\wor{i}{\cgm{\tma}{\tma}}$.
\end{itemize}
  If  $\chi$ is trivial  then
  $\wor{i}{\alpha}$.
\end{cor}
\begin{rmk}
The statement paralleling this for \cref{prop:chi_triv} doesn't seem to have a similarly clear presentation and so we have omitted it.
\end{rmk}

We now  turn to some special cases of \cref{thm:full_hy}. 
\subsection{\texorpdfstring{$A$}{A} is initial}
If $A$ is the initial object, we are in the case considered in \cref{thm:kw}.  Here 
$\hat{h}$ is a map $B\to F(\alpha)$ so that 
\[\xymatrix{\emptyset\ar[r]\ar[d]
&F(\alpha)\ar[d]
\\
B\ar[r]\ar@{.>}[ur]^{\hat{h}}&\ast}
\]
commutes.  That is, it is a map $B\to F(\alpha)$ with no additional conditions.
In \cref{thm:full_hy} the hypothesis $i_{1,\emptyset}\resizebox{.25cm}{!}{\dsquare} f(j)$ is vacuous and $\iota_1\resizebox{.25cm}{!}{\dsquare} f(j)$ becomes $i_{1,B}\resizebox{.25cm}{!}{\dsquare} f(j)$. 
Then we have the following version of \cref{thm:full_hy}.

\begin{cor}

Suppose given the solid arrows in the diagram in  \eqref{fig:intro_help} with $A$ the initial object. Assume 
$i_{1,B} \resizebox{.25cm}{!}{\dsquare} f(j)$. 

If there is an  $\hat{h}\colon B\to F(\alpha)$ 
so that 
 $\chi_{\hat{h}}$ is trivial and $\shl{i}{\cgm{\tma}{\tma}}{\Phi_{\ell(\hat{h})}}{\Psi_{\ell(\hat{h})}}h{\eqref{fig:def_big_diagram_ell}}$ 
 is nonempty then $\shl{i}{\alpha}{f}{H}{h}{\eqref{fig:intro_help}}$ is nonempty.
\end{cor}

\begin{rmk}\label{rmk:hathAinitial} If $B$ is cofibrant, the following diagram will produce a choice of $\hat{h}$.
\[\xymatrix{\emptyset\ar[r]\ar[d]
&F(\alpha)\ar[d]^{\tilde{f}(\alpha)}
\\
B\ar[r]^-f\ar@{.>}[ur]^{\hat{h}}&Y}
\]
The map $\chi$ in \cref{thm:kw}  is defined using this  diagram with  the mapping cylinder for $F(\alpha)$ and the inclusion as $\hat{h}$.
\end{rmk}

\subsection{Sections}\label{sec:sections}
The maps in a commutative diagram as in \eqref{eq:section_lift} 
\begin{equation}
    \label{eq:section_lift}
    \xymatrix{\lefts\ar[r]^-g\ar[d]^\lma&\tart\ar[d]^\tma
    \\
    \leftb\ar[r]^-f&\tarb}
\end{equation}
define maps for the diagram \eqref{fig:intro_help} if we take  $H$ to be the composite \[\cyl(\lefts)\xto{\pi}\lefts \xto{\lma}\leftb \xto{f}\tarb
=\cyl(\lefts)\xto{\pi}\lefts \xto{h}\tart\xto{\tma} \tarb.\]

\begin{cor}
Suppose the diagram in  \eqref{eq:section_lift} commutes and assume 
$\iota_1 \resizebox{.25cm}{!}{\dsquare} f(j)$ and $\iota_1\resizebox{.25cm}{!}{\dsquare} \alpha$.

If there is an  $\hat{h}\in \slt{i}{c(\alpha)\circ h}{\eqref{fig:diagram_with_terminal}}$ so that 
 $\chi_{\hat{h}}$ is trivial and $\shl{i}{\cgm{\tma}{\tma}}{\Phi_{\ell(\hat{h})}}{\Psi_{\ell(\hat{h})}}h{\eqref{fig:def_big_diagram_ell}}$ 
 is nonempty then
there is a (strict) lift  in the diagram in  \eqref{eq:section_lift}. 
\end{cor}

\begin{proof}
The hypothesis $i_{1,A}\resizebox{.25cm}{!}{\dsquare} f(j)$ of \cref{lem:lift_conditions} can be removed since  
\[\cyl(\lefts)\xto{\pi}\lefts \xto{h}\tart\xto{c(\tma)} F(\tma)\xto{\tilde{c}(j)} C(j).\]
defines a lift $\lambda$ in the diagram in  \eqref{fig:def_b}. 

Using \cref{thm:full_hy} we have maps $K$ and $g$ so that the diagram in  \eqref{fig:intro_help} commutes.  Since $\iota_1\resizebox{.25cm}{!}{\dsquare} \tma$ and the diagram in \eqref{eq:section_lift_int} commutes, the diagram in \eqref{eq:section_lift_int} has a lift $J$.
Then $\leftb\xto{i_0} \cyl(\leftb)\xto{J} \tart$ is an extension of $g$ lifting $f$ since the diagram in  \eqref{eq:expanded_section} commutes.
\end{proof}

\noindent     \begin{subequations}
    \begin{minipage}{.4\textwidth}
\begin{equation}\label{eq:section_lift_int}
\xymatrix{N_1(\lma)\ar[r]^-{H\amalg g}\ar[d]^{\iota_1}&\tart\ar[d]^\tma
\\
\cyl(\leftb)\ar[r]^-{K}&\tarb}
\end{equation}
    \end{minipage}
        \begin{minipage}{.58\textwidth}
\begin{equation}
    \xymatrix{   \lefts\ar[rd]^-{i_0}\ar[ddd]^{\lma}\ar[rr]^-h&&\tart\ar[ddd]^\tma
    \\
    &\cyl(\lefts)\ar[d]_{\cyl(\lma)}\ar[ru]^-H
    \\
    &\cyl(\leftb)\ar[dr]^{K}\ar[uru]_J
    \\
        \leftb\ar[ur]^-{i_0}\ar[rr]^-f&&\tarb
    }\label{eq:expanded_section}
\end{equation}
    \end{minipage}
    \end{subequations}

\begin{rmk}\label{rmk:hathsection} If $i\colon A\to B$ is a cofibration and $h$ is a lift of $f\circ i$, then the following diagram produces a choice of $\hat{h}$
\[\xymatrix{A\ar[r]^h\ar[d]^i&X\ar[r]^-{c(\alpha)}
&F(\alpha)\ar[d]^{\tilde{f}(\alpha)}
\\
B\ar[rr]^f\ar@{.>}[urr]^{\hat{h}}&&Y}
\]
\end{rmk}

\section{Model categories}\label{sec:model}
At its core, and similar to many other results in this paper, 
the proof of \cref{prop:top_bm}\ref{prop:top_bm1} is a strategic sequence of lifts.  We first outline those lifts in a preliminary lemma.

Let $\cocyl(Z)$ be a good cocylinder for $Z$ and $Z\xto{\mathfrak{i}}\cocyl(Z)\xto{\ev_0\times \ev_1}Z\times Z$ be the diagonal. 

\begin{lem}\label{prop:stable_cgap_nice}
Let $i\colon A\to B$ and $\theta\colon X\to Z$.  If 
\begin{itemize}
    \item $i_{0,A} \resizebox{.25cm}{!}{\dsquare} (\cocyl(Z)\xto{\ev_0\times \ev_1}Z\times Z)$, 
    \item $\iota \resizebox{.25cm}{!}{\dsquare} (\cocyl(Z)\xto{\ev_0}Z)$, and 
    \item $i\resizebox{.25cm}{!}{\dsquare} (X\times_{\theta, \ev_0}\cocyl(Z)\xto{\ev_1}Z)$,
\end{itemize}
then $\wor{i}{\theta}$.
\end{lem}

The third of these hypotheses is the most interesting since it describes compatibility between $i$ and $\theta$ in a way that generalizes the classical dimension and connectivity compatibility for spaces. 
\begin{proof}

Let $f$, $h$ and $H$ be as in the diagram in \eqref{fig:intro_help_again}.
\begin{equation}
          \xymatrix{\lefts\ar[dd]^{\lma}\ar[rr]^-{i_0}&&\cyl(\lefts)\ar[dd]|\hole\ar[dl]_-H&&\lefts\ar[dd]^{\lma}\ar[ll]_{i_1}\ar[dl]_{h}
\\
&Z&&\tart\ar[ll] _(.35)\theta
\\
\leftb\ar[ur]^f\ar[rr]^-{i_0}&&\cyl(\leftb)\ar@{.>}[ul]_{K}&&\leftb\ar[ll]_{i_1}\ar@{.>}[ul]_{{g}}}
\label{fig:intro_help_again}
\end{equation}
Using these maps  \eqref{eq:cyl_to_cocyl} commutes.   If ${\sigma}\in \slt{}{}{}{}{}{\eqref{eq:cyl_to_cocyl}}$, the diagram in \eqref{eq:cocyl_option} commutes.
\begin{equation*}
\noindent\begin{subequations}
    \begin{minipage}{.53\linewidth}
\begin{equation}\label{eq:cyl_to_cocyl}
\xymatrix{A\ar[d]^{i_0}\ar[r]^-{f\circ i} &Z\ar[r]^-{\mathfrak{i}}&\cocyl(Z)\ar[d]^{\ev_0\times\ev_1}
\\
\cyl(A)\ar@{.>}[urr]^{\sigma}\ar[rr]^-{(f\circ i\circ \pi, H)}&&Z\times Z}
\end{equation}
\end{minipage}
\begin{minipage}{.45\linewidth}
\begin{equation}\label{eq:cocyl_option}
\xymatrix@C=45pt{A\ar[r] ^-{h\times (\sigma\circ i_1)}
\ar[d]^i&X\times_{\theta,\ev_1}\cocyl(Z)\ar[d]
^{\ev_0\circ \mathrm{proj}_2}
\\
B\ar[r]^-f\ar@{.>}[ur]^{g\times \tau}&Z
}
\end{equation}
    \end{minipage}
\end{subequations}
\end{equation*}
    If 
$g\times \tau\in  \slt{}{}{}{}{}{\eqref{eq:cocyl_option}}$, the diagram in \eqref{eq:pushout_n_to_y} defines a map $N(i)\xto{\upsilon} \cocyl(Z)$ so that the diagram in  \eqref{eq:cocyl_to_cyl} commutes. 
\begin{equation*}
\begin{subequations}
\begin{minipage}{.60\linewidth}
\begin{equation}\label{eq:pushout_n_to_y}
    \xymatrix{A\amalg A\ar[r]^-{i_0\amalg i_1}\ar[d]^{i\amalg i}&\cyl(A)\ar[d]\ar@/^/[ddr]^{\sigma}
    \\
    B\amalg B\ar[r]\ar@/_/[drr]_-{(\mathfrak{i}\circ f)\amalg \tau}&N(i)\ar@{.>}[rd]^\upsilon
    \\
    &&\cocyl(Z)}
\end{equation}
    \end{minipage}
    \begin{minipage}{.38\linewidth}
\begin{equation}\label{eq:cocyl_to_cyl}
\xymatrix@C=5pt{N(i)\ar[d]^\iota\ar[rrr]^-\upsilon&&&\cocyl(Z)\ar[d]^{\ev_0}
\\
\cyl(B)\ar[rr]^-\pi\ar@{.>}[urrr]^\phi&&B\ar[r]^f&Z}
\end{equation} 
\end{minipage}
\end{subequations}
\end{equation*}
If $\phi\in \slt{}{}{}{}{}{\eqref{eq:cocyl_to_cyl}}$, then $K=\ev_1\circ \phi\colon \cyl(B)\to Z$ and  $g\colon B\to X$ make the diagram in \eqref{fig:intro_help_again} commute. 
\end{proof}

\begin{lem}\label{prop:stable_cgap_1}
If $A$ is cofibrant and $Z$ is fibrant, $i_{0,A} \resizebox{.25cm}{!}{\dsquare} (\cocyl(Z)\xto{\ev_0\times \ev_1}Z\times Z)$ (and  $\slt{}{}{}{}{}{\eqref{eq:cyl_to_cocyl}}$ is nonempty).  
If $\iota$ is a cofibration and $Z$ is fibrant,  $\iota \resizebox{.25cm}{!}{\dsquare} (\cocyl(Z)\xto{\ev_0}Z)$ (and  $\slt{}{}{}{}{}{\eqref{eq:cocyl_to_cyl}}$ is nonempty).
     
\end{lem}

\begin{proof}
    For the first statement, if $A$ is cofibrant, $A\xto{i_0} \cyl(A)$ is an acyclic cofibration.  If $Z$ is fibrant, $\cocyl(Z)\xto{\ev_0\times ev_1}Z\times Z$ is a fibration 
    
    For the second statement, 
    if $Z$ is fibrant $\cocyl(Z)\xto{\ev_0} Z$ is an acyclic  fibration.
\end{proof}

\begin{proof}[Proof of \cref{prop:top_bm}]
\cref{prop:stable_cgap_1}\label{proof:top_bm} translates the hypotheses of this statement to that of \cref{prop:stable_cgap_nice}.  That result completes the proof.
\end{proof}

In the spirit of \cref{prop:stable_cgap_1} we can also give assumptions that imply   $i\resizebox{.25cm}{!}{\dsquare} (X\times_{\theta, \ev_0}\cocyl(Z)\xto{\ev_1}Z)$.
\begin{lem}\label{prop:stable_cgap_2}
Suppose the model category is right proper.
If $i\colon A\to B$ is a cofibration
and $\theta\colon X\to Z$ is a weak equivalence between fibrant objects, then  $i\resizebox{.25cm}{!}{\dsquare} (X\times_{\theta, \ev_0}\cocyl(Z)\xto{\ev_1}Z)$ (and $\slt{}{}{}{}{}{\eqref{eq:cocyl_option}}$ is nonempty).
\end{lem}

\begin{proof}
 We have the two pullback diagrams in \eqref{eq:smaller_pullback} and \eqref{eq:larger_pullback}. 

\noindent\begin{subequations}
\begin{minipage}{.47\linewidth}
     \begin{equation}
         \xymatrix{X\times_{\theta,\ev_0}\cocyl(Z)\ar[r]^-{\mathrm{proj}_2} \ar[d]&\cocyl(Z)\ar[d]^{\ev_0}
\\
X\ar[r]^-\theta&Z}
    \label{eq:smaller_pullback}
    \end{equation}
    \end{minipage}
    \begin{minipage}{.51\linewidth}
            \begin{equation}
\xymatrix{X\times_{\theta,\ev_0}\cocyl(Z)\ar[r]\ar[d]^{\id\times\ev_1}&\cocyl(Z)\ar[d]^{\ev_0\times \ev_1}
\\
X\times Z\ar[r]^-{\theta\times \id}&Z\times Z
}
\label{eq:larger_pullback}
\end{equation}
\end{minipage}
\end{subequations}

\noindent The composites in the diagrams in  \eqref{eq:acyclic_fibration_map} and \eqref{eq:acyclic_fibration_map_2} are the same map.
\\
\noindent \begin{subequations}
    \begin{equation}\label{eq:acyclic_fibration_map}
X\times_{\theta,\ev_0} \cocyl(Z)\xto{\mathrm{proj}_2}\cocyl (Z)\xrightarrow{\ev_1} Z
\end{equation}
\begin{equation}\label{eq:acyclic_fibration_map_2}
X\times_{\theta,\ev_0}\cocyl(Z)\xto{\id\times\ev_1}  X\times Z\xto{\mathrm{proj}_2} Z\end{equation}
\end{subequations}

If $Z$ is fibrant, the maps $\ev_0,\ev_1\colon \cocyl(Z)\to Z$ are acyclic fibrations.  The model category is right proper and $\theta$ is an weak equivalence,  
so the top horizontal  map 
in the diagram in \eqref{eq:smaller_pullback} is a weak equivalence.  Composing with $\ev_1$
implies 
\eqref{eq:acyclic_fibration_map}
is a weak equivalence.

The first map in \eqref{eq:acyclic_fibration_map_2} is the left vertical map in the diagram in  \eqref{eq:larger_pullback}.  It is the pullback of a fibration and so is a fibration. The second map in \eqref{eq:acyclic_fibration_map_2} is the projection.   It is a fibration since $X$ is fibrant.
Therefore \eqref{eq:acyclic_fibration_map_2} is a fibration.  

Since 
\eqref{eq:acyclic_fibration_map} and  \eqref{eq:acyclic_fibration_map_2}  are the same map it is an acylic fibration.
\end{proof}

\begin{rmk}
The diagram in \eqref{eq:cocyl_option} is the right homotopy version of the diagram in  \eqref{fig:intro_help} and the results of this paper could be reworked to prioritize the diagram in  \eqref{eq:cocyl_option} over the diagram in  \eqref{fig:intro_help}.  
\end{rmk}

\bibliographystyle{annote2}
\bibliography{Refs}%

\providecommand{\bysame}{\leavevmode\hbox to3em{\hrulefill}\thinspace}
\providecommand{\MR}{\relax\ifhmode\unskip\space\fi MR }
\providecommand{\MRhref}[2]{%
  \href{http://www.ams.org/mathscinet-getitem?mr=#1}{#2}
}
\providecommand{\doi}[1]{%
  doi:\href{https://dx.doi.org/#1}{#1}}
\providecommand{\arxiv}[1]{%
  arXiv:\href{https://arxiv.org/abs/#1}{#1}}
\providecommand{\href}[2]{#2}
\begin{thebibliography}{ABFJ20}

\bibitem[ABFJ20]{AGFJ}
M.~Anel, G.~Biedermann, E.~Finster, and A.~Joyal, \emph{A generalized
  {B}lakers-{M}assey theorem}, J. Topol. \textbf{13} (2020), no.~4, 1521--1553.
  \doi{10.1112/topo.12163} \arxiv{1703.09050}

\bibitem[Bro71]{brown}
R.~F. Brown, \emph{The {L}efschetz fixed point theorem}, Scott, Foresman and
  Co., Glenview, Ill.-London, 1971.

\bibitem[CSW16]{CSW}
W.~Chach\'{o}lski, J.~Scherer, and K.~Werndli, \emph{Homotopy excision and
  cellularity}, Ann. Inst. Fourier (Grenoble) \textbf{66} (2016), no.~6,
  2641--2665. \arxiv{1408.3252}

\bibitem[CDI04]{CDI}
J.~D. Christensen, W.~G. Dwyer, and D.~C. Isaksen, \emph{Obstruction theory in
  model categories}, Adv. Math. \textbf{181} (2004), no.~2, 396--416.
  \arxiv{math/0109170}

\bibitem[Cou09]{coufal}
V.~Coufal, \emph{A parametrized fixed point theorem}, Proc. Amer. Math. Soc.
  \textbf{137} (2009), no.~11, 3939--3942. \doi{10.1090/S0002-9939-09-09978-X}

\bibitem[Dot16]{dotto}
E.~Dotto, \emph{Equivariant diagrams of spaces}, Algebr. Geom. Topol.
  \textbf{16} (2016), no.~2, 1157--1202. \doi{10.2140/agt.2016.16.1157}
  \arxiv{1502.05725}

\bibitem[HQ74]{hq}
A.~Hatcher and F.~Quinn, \emph{Bordism invariants of intersections of
  submanifolds}, Trans. Amer. Math. Soc. \textbf{200} (1974), 327--344.
  \doi{10.2307/1997261}

\bibitem[Hov99]{hovey}
M.~Hovey, \emph{Model categories}, Mathematical Surveys and Monographs,
  vol.~63, American Mathematical Society, Providence, RI, 1999.

\bibitem[Hus82]{husseini}
S.~Y. Husseini, \emph{Generalized {L}efschetz numbers}, Trans. Amer. Math. Soc.
  \textbf{272} (1982), no.~1, 247--274. \doi{10.2307/1998959}

\bibitem[KW07]{kw}
J.~R. Klein and E.~B. Williams, \emph{Homotopical intersection theory. {I}},
  Geom. Topol. \textbf{11} (2007), 939--977. \doi{10.2140/gt.2007.11.939}
  \arxiv{math/0512479}

\bibitem[KW10]{kw2}
\bysame, \emph{Homotopical intersection theory. {II}. {E}quivariance}, Math. Z.
  \textbf{264} (2010), no.~4, 849--880. \doi{10.1007/s00209-009-0491-1}
  \arxiv{0803.0017}

\bibitem[May99]{concise}
J.~P. May, \emph{A concise course in algebraic topology}, Chicago Lectures in
  Mathematics, University of Chicago Press, Chicago, IL, 1999.

\bibitem[Spa95]{spanier}
E.~H. Spanier, \emph{Algebraic topology}, Springer-Verlag, New York, [1995],
  Corrected reprint of the 1966 original.

\bibitem[Sun13]{sun}
G.~Sunyeekhan, \emph{A bordism viewpoint of fiberwise intersections}, Int. J.
  Math. Math. Sci. (2013), Art. ID 430105, 6. \doi{10.1155/2013/430105}
  \arxiv{1210.4665}

\bibitem[tD08]{td}
T.~tom Dieck, \emph{Algebraic topology}, EMS Textbooks in Mathematics, European
  Mathematical Society (EMS), Z\"{u}rich, 2008. \doi{10.4171/048}

\bibitem[Whi78]{whitehead}
G.~W. Whitehead, \emph{Elements of homotopy theory}, Graduate Texts in
  Mathematics, vol.~61, Springer-Verlag, New York-Berlin, 1978.

\end{thebibliography}

\end{document}